\documentclass{amsproc}
\usepackage{epsfig,amscd,amssymb,amsmath,amsfonts}
\usepackage[margin=1.05in]{geometry}
\usepackage[colorlinks]{hyperref}
\usepackage{enumitem}
\usepackage{xcolor}
\newtheorem{theorem}{Theorem}[section]
\newtheorem{lemma}[theorem]{Lemma}

\theoremstyle{definition}

\newtheorem{example}[theorem]{Example}

\theoremstyle{remark}
\newtheorem{remark}[theorem]{Remark}

\numberwithin{equation}{section}

\begin{document}

\title{A Geometric Application for the $det^{S^2}$ Map}

\author{Mihai D. Staic}
\address{Department of Mathematics and Statistics, Bowling Green State University, Bowling Green, OH 43403 }
\address{Institute of Mathematics of the Romanian Academy, PO.BOX 1-764, RO-70700 Bu\-cha\-rest, Romania.}
\email{mstaic@bgsu.edu}

\author{Jacob Van Grinsven}
\address{Department of Mathematics and Statistics, Bowling Green State University, Bowling Green, OH 43403 }
\email{jdvangr@bgsu.edu}






\begin{abstract}
We discuss properties of the $det^{S^2}$ map, present a few explicit computations, and give a geometrical interpretation for the condition $det^{S^2}((v_{i,j})_{1\leq i<j\leq 4})=0$. 
\end{abstract}

\maketitle

\section{Introduction} 

The determinant of an $n\times n$ matrix has countless applications in mathematics, so it is not surprising that attempts to generalize it have been made over the years.  A modern approach to this problem is presented  in \cite{gkz} where the notion of hyperdeterminant is discussed in detail, but similar ideas can be traced back to Cayley's work. 

The $det^{S^2}$  map (for a vector space of dimension $d=2$) was introduced in \cite{sta2} as the unique multilinear function that satisfies a certain universality property. This map does not seem to fit in the setting of hyperdeterminants  described in  \cite{gkz}, but just like the classical determinant, the $det^{S^2}$ map is associated to an exterior algebra "`like"' construction (more precisely an exterior GSC-operad). The construction of the GSC operad was inspired by results on Higher Hochschild homology (\cite{cs}, \cite{p}), and the Swiss-Cheese Operad (\cite{vo}). Even though the operad construction exists for  vector spaces of any dimension, the  existence of  the $det^{S^2}$  map  is not known for dimension $d>2$ (but it was conjectured in \cite{sta2}).

In this paper we outline a series of properties and results  about $det^{S^2}$. In section 2 we recall a few notations and formulas from \cite{sta2}. In section 3 we show that $det^{S^2}$ is invariant under the action of the group $SL_2(k)$, and under the action of the symmetric group  $S_4$. We also give some explicit computations with geometrical flavor. 

In section 4 we present an analog of the  well known fact that the determinant of a square matrix vanishes if and only if the column vectors are linearly dependent. For $(v_{i,j})_{1\leq i<j\leq 4}$ a collection of six vectors, we describe the relationship between $\det^{S^2}((v_{i,j})_{1\leq i<j\leq 4})$ vanishing and the existence of a quadrilateral $Q_1Q_2Q_3Q_4$ with edges $\overrightarrow{Q_iQ_j}$ a multiple of $v_{i,j}$. 
We conclude the paper with a few remarks and examples.

\section{Preliminaries} 
 
In this paper $k$ is a field with $char(k)=0$, and $V$ is a $k$-vector space of dimension $2$. For geometrical applications we will take $k=\mathbb{R}$. Some of the results presented here work also over a commutative ring. 

The following convention was used in \cite{sta2} to represent an element from $V^6$. Consider $v_{i,j}\in V$ for all $1\leq i <j\leq 4$,  we denote
$$\mathfrak{V}=(v_{i,j})_{1\leq i<j\leq 4}=\begin{pmatrix}
0 & v_{1,2}& v_{1,3}& v_{1,4}\\
 & 0 & v_{2,3}& v_{2,4}\\
 & &0& v_{3,4}\\
 & & &0
\end{pmatrix} \in V^6. $$
This notation is convenient when we want to keep track of the positions of the elements in $\mathfrak{V}$. 
The zeros on the diagonal of $\mathfrak{V}$ do not play any role, they are mostly for symmetry and help keep track of rows and columns. When there is no danger of confusion we will use the notation $\mathfrak{V}$ for a generic element in $V^6$. 

Notice that we have an natural action of the symmetric group $S_4$ on $V^6$ given by 
$$\sigma\cdot\begin{pmatrix}
0 & v_{1,2}& v_{1,3}& v_{1,4}\\
 & 0 & v_{2,3}& v_{2,4}\\
 & &0& v_{3,4}\\
 & & &0
\end{pmatrix}=\begin{pmatrix}
0 & v_{\sigma(1),\sigma(2)}& v_{\sigma(1),\sigma(3)}& v_{\sigma(1),\sigma(4)}\\
 & 0 & v_{\sigma(2),\sigma(3)}& v_{\sigma(2),\sigma(4)}\\
 & &0& v_{\sigma(3),\sigma(4)}\\
 & & &0
\end{pmatrix}
$$
with the conventions that  $v_{i,j}=v_{j,i}$ for $i>j$.

We recall from \cite{sta2} the formula for the map $det^{S^2}:V^6\to k$. For $v_{i,j}=(\alpha_{i,j},\beta_{i,j})\in k^2$ we have
\begin{eqnarray*} &det^{S^2}\begin{pmatrix} 
0& (\alpha_{1,2},\beta_{1,2})&(\alpha_{1,3},\beta_{1,3})&(\alpha_{1,4},\beta_{1,4})\\
&0&(\alpha_{2,3},\beta_{2,3})&(\alpha_{2,4},\beta_{2,4})\\
& &0&(\alpha_{3,4},\beta_{3,4})\\
&&&0
\end{pmatrix}=&\\
&\alpha_{1,2}\alpha_{2,3}\alpha_{3,4}\beta_{1,3}\beta_{2,4}\beta_{1,4}+
\alpha_{1,2}\beta_{2,3}\alpha_{3,4}\beta_{1,3}\beta_{2,4}\alpha_{1,4}+
\alpha_{1,2}\beta_{2,3}\beta_{3,4}\alpha_{1,3}\alpha_{2,4}\beta_{1,4}&\\
&+\beta_{1,2}\beta_{2,3}\alpha_{3,4}\alpha_{1,3}\alpha_{2,4}\beta_{1,4}+
\beta_{1,2}\alpha_{2,3}\beta_{3,4}\beta_{1,3}\alpha_{2,4}\alpha_{1,4}+
\beta_{1,2}\alpha_{2,3}\beta_{3,4}\alpha_{1,3}\beta_{2,4}\alpha_{1,4}&\\
&-\beta_{1,2}\beta_{2,3}\beta_{3,4}\alpha_{1,3}\alpha_{2,4}\alpha_{1,4}-
\beta_{1,2}\alpha_{2,3}\beta_{3,4}\alpha_{1,3}\alpha_{2,4}\beta_{1,4}-
\beta_{1,2}\alpha_{2,3}\alpha_{3,4}\beta_{1,3}\beta_{2,4}\alpha_{1,4}&\\
&-\alpha_{1,2}\alpha_{2,3}\beta_{3,4}\beta_{1,3}\beta_{2,4}\alpha_{1,4}-
\alpha_{1,2}\beta_{2,3}\alpha_{3,4}\alpha_{1,3}\beta_{2,4}\beta_{1,4}-
\alpha_{1,2}\beta_{2,3}\alpha_{3,4}\beta_{1,3}\alpha_{2,4}\beta_{1,4}.&
\end{eqnarray*}
This formula was obtained from a exterior algebra "like" construction. Essentially, $det^{S^2}$ is a the unique nontrivial multilinear map defined on $V^6$ which has the property that $det^{S^2}\begin{pmatrix}
0 & v_{1,2}& v_{1,3}& v_{1,4}\\
 & 0 & v_{2,3}& v_{2,4}\\
 & &0& v_{3,4}\\
 & & &0
\end{pmatrix}=0$  if there exists $1\leq i<j<k\leq 4$ such that $v_{i,j}=v_{i,k}=v_{j,k}$. We refer to \cite{sta2} for more details. 

\begin{remark} One should notice that the $det^{S^2}$ map does not coincide with any of the hyperdeterminant  maps discussed in \cite{gkz}. Due to dimension restrictions, the only possible candidate is the hyperdeterminant of multidimensional the matrices of type $2\times 2\times 3$ (see \cite{gkz} page 463). However if in the formula of the hyperdeterminant we take $a_{0,0,0}=a_{1,0,1}=a_{0,1,1}=a_{1,1,2}=1$ and all of the other entries $a_{i,j,k}=0$ then $Det(A)=-1$, while our $det^{S^2}$ map is automatically $0$ if only four of the twelve scalar entries in $\mathfrak{V}$ are nonzero. 
\end{remark}

\section{Properties and Computations}

In this section we establish some basic properties of $det^{S^2}$, and give a few  explicit computations. Most of these results  can be checked by direct computation, but they also follow from Lemma \ref{lemma1}.

With the notations from the previous section we have the following
\begin{lemma} 
(a) If $T:V\to V$ is a linear map then $$det^{S^2}\begin{pmatrix}
0 & T(v_{1,2})& T(v_{1,3})& T(v_{1,4})\\
 & 0 & T(v_{2,3})& T(v_{2,4})\\
 & &0& T(v_{3,4})\\
 & & &0
\end{pmatrix}=det(T)^3det^{S^2}\begin{pmatrix}
0 & v_{1,2}& v_{1,3}& v_{1,4}\\
 & 0 & v_{2,3}& v_{2,4}\\
 & &0& v_{3,4}\\
 & & &0
\end{pmatrix}.$$ In particular $det^{S^2}$ is invariant under the action of the group $SL_2(k)$. \\
(b) For any $\sigma\in S_4$ we have
$$det^{S^2}\left(\sigma\cdot\begin{pmatrix}
0 & v_{1,2}& v_{1,3}& v_{1,4}\\
 & 0 & v_{2,3}& v_{2,4}\\
 & &0& v_{3,4}\\
 & & &0
\end{pmatrix}\right)=sgn(\sigma)det^{S^2}\begin{pmatrix}
0 & v_{1,2}& v_{1,3}& v_{1,4}\\
 & 0 & v_{2,3}& v_{2,4}\\
 & &0& v_{3,4}\\
 & & &0
\end{pmatrix}.$$
\end{lemma}
\begin{proof} This follows by direct computations.  
\end{proof}

\begin{example}
Let $P_1$, $P_2$, $P_3$, $P_4$ be four points in the plane $\mathbb{R}^2$, and take  $v_{i,j}=\overrightarrow{P_iP_j}$ then 
$$det^{S^2}\begin{pmatrix}
0 & v_{1,2}& v_{1,3}& v_{1,4}\\
 & 0 & v_{2,3}& v_{2,4}\\
 & &0& v_{3,4}\\
 & & &0
\end{pmatrix}=0.$$
The converse of this example will be discussed in detail in the next section. \label{example1}
\end{example}

\begin{example}
Let $L_1,L_2,L_3,L_4$ be lines  in $\mathbb{R}^2$ passing through the origin, with $\theta_i$ the angle between the positive $x$-axis and $L_i$. If $v_{i,j}=\begin{pmatrix}
\cos(\theta_j-\theta_i)\\
\sin(\theta_j-\theta_i)\\
\end{pmatrix}$ then $$det^{S^2}\begin{pmatrix}
0 & v_{1,2}& v_{1,3}& v_{1,4}\\
 & 0 & v_{2,3}& v_{2,4}\\
 & &0& v_{3,4}\\
 & & &0
\end{pmatrix}=\sin(\theta_2-\theta_1)\sin(2(\theta_3-\theta_2))\sin(\theta_4-\theta_3).$$ 
Notice that because $det^{S^2}$ is a geometric invariant the result depends only on the angle between the $L_{i}$ and $L_{i+1}$, so if we take $\phi_i=\theta_{i+1}-\theta_i$ we have 
$$det^{S^2}\begin{pmatrix}
0 & v_{1,2}& v_{1,3}& v_{1,4}\\
 & 0 & v_{2,3}& v_{2,4}\\
 & &0& v_{3,4}\\
 & & &0
\end{pmatrix}=\sin(\phi_1)\sin(2\phi_2)\sin(\phi_3).$$
\end{example}

One can also check the following description for $det^{S^2}$.
\begin{remark}
Take  $(v_{i,j})_{1\leq i<j\leq 4}\in V^6$, and let $\langle \_,\_\rangle$ be the standard inner product on $\mathbb{R}^2$. Then
\begin{equation*}
    \begin{split}
        det^{S^2}((v_{i,j})_{1\leq i<j\leq 4})=&
        det(v_{1,4},v_{2,4})\langle v_{1,2},v_{3,4}\rangle\langle v_{1,3},v_{2,3}\rangle\\
        &+det(v_{3,4},v_{1,4})\langle v_{1,3},v_{2,4}\rangle\langle v_{1,2},v_{2,3}\rangle\\
        &+det(v_{2,4},v_{3,4})\langle v_{1,4},v_{2,3}\rangle\langle v_{1,2},v_{1,3}\rangle.\\
    \end{split}
\end{equation*}
Where $det(v,w)$ is the determinant of the $2\times2$ matrix with $v$ and $w$ as columns. In particular, we have $$det^{S^2}\begin{pmatrix}
0 & v_1& v_2& v_3\\
 & 0 & v_3^\perp& v_2^\perp\\
 & &0& v_1^\perp\\
 & & &0
\end{pmatrix}=0$$
for any $v_1,v_2,v_2\in \mathbb{R}^2$ and $v_i^\perp$ any non-zero vector orthogonal to $v_i$.
\end{remark}

\section{Main Result}
In this section we present the converse of the Example \ref{example1}. First we need the following lemma. 

\begin{lemma} Let $(v_{i,j})_{1\leq i<j\leq 4}\in V^6$, then 
$$det^{S^2}((v_{i,j})_{1\leq i<j\leq 4})=det\begin{pmatrix}
\alpha_{1,2} & \alpha_{2,3} & 0 & -\alpha_{1,3} & 0 &0\\
\beta_{1,2} & \beta_{2,3} & 0 &-\beta_{1,3} & 0&0\\
\alpha_{1,2} & 0 & 0 & 0 & \alpha_{2,4} & -\alpha_{1,4}\\
\beta_{1,2} & 0 & 0 & 0 & \beta_{2,4} & -\beta_{1,4}\\
0 & 0 & \alpha_{3,4} & \alpha_{1,3} & 0 &-\alpha_{1,4}\\
0 & 0 & \beta_{3,4} & \beta_{1,3} & 0 &-\beta_{1,4}
\end{pmatrix}.$$

\label{lemma1}
\end{lemma}
\begin{proof}
It follows by direct computations. 
\end{proof}

We are now ready to prove the main result of this note. 
\begin{theorem} Take $V=\mathbb{R}^2$, and let $(v_{i,j})_{1\leq i<j\leq 4}\in V^6$, then the following are equivalent.\\ 
(a) $det^{S^2}\begin{pmatrix}
0 & v_{1,2}& v_{1,3}& v_{1,4}\\
 & 0 & v_{2,3}& v_{2,4}\\
 & &0& v_{3,4}\\
 & & &0
\end{pmatrix}=0$.\\
(b) There exist points $Q_1$, $Q_2$, $Q_3$, $Q_4$ in the plane $\mathbb{R}^2$, and $\lambda_{i,j}\in\mathbb{R}$ for  $1\leq i<j\leq 4$ not all zero such that $\lambda_{i,j}v_{i,j}=\overrightarrow{Q_iQ_j}$. 
\label{th1}
\end{theorem}
\begin{proof}
First notice that assuming $(b)$, from  Example \ref{example1} we know that 
$$det^{S^2}\begin{pmatrix}
0 & \lambda_{1,2}v_{1,2}& \lambda_{1,3}v_{1,3}& \lambda_{1,4}v_{1,4}\\
 & 0 & \lambda_{2,3}v_{2,3}& \lambda_{2,4}v_{2,4}\\
 & &0& \lambda_{3,4}v_{3,4}\\
 & & &0
\end{pmatrix}=0.$$
However we cannot conclude $(a)$ since since we don't know that all $\lambda_{i,j}$ are nonzero. 

Regardless, condition $(b)$ is satisfied if and only if there is a non-trivial solution $(\lambda_{i,j})_{1\leq i<j\leq 4}$ to the system of vector equations
\begin{eqnarray}
\begin{cases}
\lambda_{1,2}v_{1,2}+\lambda_{2,3}v_{2,3}-\lambda_{1,3}v_{1,3}=0\\
\lambda_{1,2}v_{1,2}+\lambda_{2,4}v_{2,4}-\lambda_{1,4}v_{1,4}=0\\
\lambda_{1,3}v_{1,3}+\lambda_{3,4}v_{3,4}-\lambda_{1,4}v_{1,4}=0\\
\lambda_{2,3}v_{2,3}+\lambda_{3,4}v_{3,4}-\lambda_{2,4}v_{2,4}=0.
\end{cases} \label{equation2}
\end{eqnarray}
This induces a system of linear equations described by 
\begin{equation}
    \begin{pmatrix}
\alpha_{1,2} & \alpha_{2,3} & 0 & -\alpha_{1,3} & 0 &0\\
\beta_{1,2} & \beta_{2,3} & 0 &-\beta_{1,3} & 0&0\\
\alpha_{1,2} & 0 & 0 & 0 & \alpha_{2,4} & -\alpha_{1,4}\\
\beta_{1,2} & 0 & 0 & 0 & \beta_{2,4} & -\beta_{1,4}\\
0 & 0 & \alpha_{3,4} & \alpha_{1,3} & 0 &-\alpha_{1,4}\\
0 & 0 & \beta_{3,4} & \beta_{1,3} & 0 &-\beta_{1,4}\\
0 & \alpha_{2,3} &\alpha_{3,4} & 0 & -\alpha_{2,4} & 0\\
0 & \beta_{2,3} & \alpha_{3,4} & 0 & -\beta_{2,4} & 0
\end{pmatrix}\begin{pmatrix}
\lambda_{1,2}\\
\lambda_{2,3}\\
\lambda_{3,4}\\
\lambda_{1,3}\\
\lambda_{2,4}\\
\lambda_{1,4}
\end{pmatrix}=0.\label{equation1}
\end{equation}
And so, condition $(b)$ is equivalent to the above matrix being rank deficient (i.e. rank $5$ or smaller).

Denote the row vectors in the matrix from equation \ref{equation1} as $\{R_i\}_{1\leq i\leq 8}$. The matrix is rank deficient if and only if the submatrices formed by each choice of $6$ rows from $\{R_i\}_{1\leq i\leq 8}$ results in a zero determinant. Notice the set $\{R_1,R_3,R_5,R_7\}$ is linearly dependent. This is true about the set $\{R_2,R_4,R_6,R_8\}$ as well. This shows that the only possibility of a non-zero 6 by 6 determinant arises if we choose $3$ rows from each of the two sets $\{R_1,R_3,R_5,R_7\}$ and $\{R_2,R_4,R_6,R_8\}$.

We want to show that in order to compute its rank we can chose the first six rows of the above matrix. Suppose we have a submatrix with 3 rows from $\{R_1,R_3,R_5,R_7\}$ and 3 rows from $\{R_2,R_4,R_6,R_8\}$. If $R_1$ is not chosen as one of the rows then $R_3,R_5,R_7$ must be chosen, and the linear dependence relation on the set $\{R_1,R_3,R_5,R_7\}$ gives a sequence of elementary row operations that replaces the row $R_7$ with $R_1$. The same argument allows us to replace $R_7$ with whichever row of the set $\{R_1,R_3,R_5\}$ is missing. A similar argument on the set $\{R_2,R_4,R_6,R_8\}$ means we only need to consider the rows $R_2,R_4,R_6$ as being rows in our submatrix. Thus the determinant of any such matrix is a nonzero multiple the determinant of the matrix consisting of rows $\{R_1,R_2,R_3,R_4,R_5,R_6\}$. Lemma \ref{lemma1} now shows that the determinant of this matrix is given by $det^{S^2}\begin{pmatrix}
0 & v_{1,2}& v_{1,3}& v_{1,4}\\
& 0 & v_{2,3}& v_{2,4}\\
& &0& v_{3,4}\\
& & &0
\end{pmatrix}$.

Thus the $6\times 8$ matrix from equation \ref{equation1} is rank deficient if and only if $det^{S^2}\begin{pmatrix}
0 & v_{1,2}& v_{1,3}& v_{1,4}\\
& 0 & v_{2,3}& v_{2,4}\\
& &0& v_{3,4}\\
& & &0
\end{pmatrix}=0$, which shows the equivalence of $(a)$ and $(b)$. 
\end{proof}

\begin{remark} Geometrically, Theorem \ref{th1} says that the vectors $v_{i,j}$ give the directions of all diagonals in a quadrilateral if and only if $det^{S^2}((v_{i,j})_{1\leq i<j\leq 4})=0$. 
\end{remark}

\begin{remark}
It should be noted that different values of $det^{S^2}$ can come from different orderings of the same 6 vectors.  The orbit of $(v_{i,j})_{1\leq i<j\leq 4}$  under the action of the group $S_4$ is bounded above by $|S_4|=24$, but the total number of arrangements of 6 vectors is $720$. Thus we cannot relax condition $(b)$ in Theorem \ref{th1} to be $\{\lambda_{i,j}v_{i,j}\;\vert\;1\leq i<j\leq 4\}=\left\{\overrightarrow{Q_iQ_j} \;\vert\; 1\leq i<j\leq 4\right\}$ as sets.

For example if $v_{1,2}=v_{2,3}=v_{1,3}=e_1$ and $v_{1,4}=v_{2,4}=v_{3,4}=e_2$ then $\lambda_{1,2}=\lambda_{2,3}=\lambda_{1,3}=0$ and $\lambda_{1,4}=\lambda_{2,4}=\lambda_{3,4}=1$ is a solution for the system \ref{equation2}. On the other hand if $w_{1,2}=w_{2,3}=w_{3,4}=e_1$ and $w_{1,3}=w_{2,4}=w_{1,4}=e_2$ we have that 
$det^{S^2}\begin{pmatrix}
0 & w_{1,2}& w_{1,3}& w_{1,4}\\
& 0 & w_{2,3}& w_{2,4}\\
& &0& w_{3,4}\\
& & &0
\end{pmatrix}=1$ and so the equation \ref{equation2} has no non-trivial solution. Even though $\{v_{i,j}\;\vert\;1\leq i<j\leq 4\}=\{w_{i,j}\;\vert\;1\leq i<j\leq 4\}$ as sets, the $det^{S^2}$ map takes different values on the corresponding elements in $V^6$. 
\end{remark}

\begin{remark} From a homological perspective, Theorem \ref{th1} can be seen as studying  relations among relations. More precisely, if we denote $R_{i,j,k}=\lambda_{i,j}v_{i,j}+\lambda_{j,k}v_{j,k}-\lambda_{i,k}v_{i,k}$, then we have already imposed  the condition $R_{1,2,3}-R_{1,2,4}+R_{1,3,4}-R_{2,3,4}=0$. This is very similar with Cayley's theory of elimination (see Appendix B in \cite{gkz}). 
It is natural to ask if this idea can be used to define a $det^{S^2}$ map for any finite dimensional vector space $V$. 
\end{remark}



\bibliographystyle{amsalpha}

\begin{thebibliography}{A}




\bibitem 
{cs}
S. Carolus,  and M. D. Staic, 
\textit{$G$-Algebra Structure on the Higher Order Hochschild Cohomology $H^*_{S^2}(A,A)$}, to appear in Algebra Colloquium,  arXiv:1804.05096

\bibitem
{gkz}
I. M. Gelfand, M. M. Kapranov,  and A. V. Zelevinsky,    \textit{Discriminants, resultants, and multidimensional determinants}, Birkhäuser Boston,  (1994).



\bibitem
{p} T. Pirashvili, \textit{Hodge decomposition for higher order Hochschild homology}, Ann. Sci. Ecole Norm. Sup., (4) {\bf 33} (2000), 151--179.


\bibitem
{sta2} M. D. Staic, \textit{The Exterior Graded Swiss-Cheese Operad $\Lambda^{S^2}(V)$ (with an appendix by Ana Lorena Gherman and Mihai D. Staic)}, 	arXiv:2002.00520.


\bibitem
{vo}
 A. A. Voronov, \textit{The Swiss-Cheese Operad},  Contemporary Mathematics, {\bf 239} (1999), 365--373.



\end{thebibliography}

\end{document}